\documentclass{amsart}
\usepackage{amscd,graphicx,epsfig}
\usepackage{amssymb}

\title[Automorphisms of the Lie Algebra of Vector Fields on Affine $n$-Space]
{Automorphisms of the Lie Algebra of Vector Fields on Affine $n$-Space}
\author{Hanspeter Kraft and Andriy Regeta}
\date{February 2014}
\thanks{The authors are partially supported by a grant from the SNF (Schweizerischer Nationalfonds)}

\address{Mathematisches Institut,
Universit\"at Basel, Rheinsprung 21, CH-4051 Basel}
\email{Hanspeter.Kraft@unibas.ch}
\email{Andriy.Regeta@unibas.ch}

\newtheorem{thm}{Theorem}[section]
\newtheorem*{thm*}{Theorem}
\newtheorem*{conj*}{Conjecture}
\newtheorem*{mthm}{Main Theorem}

\newtheorem*{prob*}{Problem}
\newtheorem*{satz*}{Satz}

\newtheorem{prop}[thm]{Proposition}
\newtheorem*{prop*}{Proposition}
\newtheorem{lem}[thm]{Lemma}
\newtheorem*{lem*}{Lemma}
\newtheorem{cor}[thm]{Corollary}
\newtheorem*{cor*}{Corollary}

\theoremstyle{definition}

\theoremstyle{remark}
\newtheorem*{rem*}{Remark}
\newtheorem{rem}[thm]{Remark}

\newcommand{\di}{\partial_{x_{i}}}

\newcommand{\fp}[2]{\frac{\partial#1}{\partial x_{#2}}}

\renewcommand{\u}{\mathbf u}
\renewcommand{\t}{\mathbf t}

\newcommand{\name}[1]{\textsc{#1\/}}

\newcommand{\CC}{{\mathbb C}}

\newcommand{\QQ}{{\mathbb Q}}

\newcommand{\An}{{\mathbb A}^{n}}

\newcommand{\simto}{\xrightarrow{\sim}}
\newcommand{\be}{\begin{enumerate}}
\newcommand{\ee}{\end{enumerate}}
\newcommand{\beq}{\begin{equation}}
\newcommand{\eeq}{\end{equation}}

\DeclareMathOperator{\id}{id}

\DeclareMathOperator{\SL}{SL}

\DeclareMathOperator{\Aut}{Aut}
\DeclareMathOperator{\SAut}{SAut}

\DeclareMathOperator{\Ad}{Ad}

\DeclareMathOperator{\VF}{Vec}
\DeclareMathOperator{\GL}{GL}
\DeclareMathOperator{\Lie}{Lie}
\newcommand{\VFc}{\VF^{c}}
\newcommand{\VFzero}{\VF^{0}}
\DeclareMathOperator{\Der}{Der}
\DeclareMathOperator{\Jac}{Jac}

\DeclareMathOperator{\Div}{Div}
\DeclareMathOperator{\cent}{\mathfrak{cent}}

\DeclareMathOperator{\ad}{ad}

\DeclareMathOperator{\Gal}{Gal}
\DeclareMathOperator{\res}{res}
\DeclareMathOperator{\gl}{\mathfrak{gl}}
\DeclareMathOperator{\sll}{\mathfrak{sl}}

\newcommand{\gln}{\gl_{n}}
\newcommand{\sln}{\sll_{n}}

\newcommand{\bbmat}{\begin{bmatrix}}
\newcommand{\ebmat}{\end{bmatrix}}
\newcommand{\bsmat}{\begin{smallmatrix}}
\newcommand{\esmat}{\end{smallmatrix}}

\DeclareMathOperator{\aff}{\mathfrak{aff}}
\DeclareMathOperator{\saff}{\mathfrak{saff}}

\DeclareMathOperator{\Aff}{Aff}
\DeclareMathOperator{\SAff}{SAff}

\DeclareMathOperator{\Adj}{Adj}

\newcommand{\ef}{\partial_{E}}

\newcommand{\s}{\mathfrak{s}}

\newcommand{\Kx}{K[x_{1},\ldots,x_{n}]}
\newcommand{\dxi}{\frac{\partial}{\partial{x_{i}}}}
\DeclareMathOperator{\AutL}{Aut_{Lie}}
\DeclareMathOperator{\tr}{tr}
\newcommand{\dx}[1]{\frac{\partial}{\partial{x_{#1}}}}
\newcommand{\dd}[1]{\partial_{x_{#1}}}

\renewcommand{\phi}{\varphi}
\def \itt #1,#2:{\medskip\item[$\bullet$] %
     page\ \ignorespaces#1, line\ \ignorespaces#2:\ \ignorespaces}

\newcommand{\lab}[1]{\label{#1}}

\DeclareMathOperator{\rk}{rank}

\begin{document}

\begin{abstract} 
We show that every Lie algebra automorphisms of the vector fields $\VF(\An)$ of affine $n$-space $\An$, of the vector fields $\VFc(\An)$ with constant divergence, and of the vector fields $\VFzero(\An)$ with divergence zero is induced by an automorphism of $\An$. This generalizes results of the second author obtained in dimension 2, see \cite{Re2013Lie-subalgebras-of}. The case of $\VF(\An)$ is due to \name{Bavula} \cite{Ba2013The-group-of-autom}.

As an immediate consequence, we get the following result due to \name{Kulikov} \cite{Ku1992Generalized-and-lo}. If every injective endomorphism of the Lie algebra $\VF(\An)$ is an automorphism, then the Jacobian Conjecture holds in dimension $n$.

\end{abstract}

\maketitle

\section{Introduction}
Let $K$ be an algebraically closed field of characteristic zero. 
Denote by $\VF(\An)$ the Lie algebra of polynomial vector fields on affine $n$-space $\An = K^{n}$. We have the standard identifications
$$
\VF(\An) = \Der (\Kx) = \left\{\sum_{i}f_{i}\dxi \mid f_{i}\in \Kx\right\}.
$$
The group $\Aut(\An)$ of polynomial automorphisms of $\An$ acts  on $\VF(\An)$ in the usual way. For $\phi\in\Aut(\An)$ and $\delta\in\VF(\An)$ we define
$$
\Ad(\phi) \delta := {\phi^{*}}^{-1} \circ \delta \circ \phi^{*}
$$ 
where we consider $\delta$ as a derivation $\delta\colon\Kx \to \Kx$ and where $\phi^{*}\colon \Kx \to \Kx$, $f\mapsto f\circ\phi$,  is the co-morphism of $\phi$.
It is shown in \cite{Ba2013The-group-of-autom} that $\Ad\colon \Aut(\An) \to \AutL(\VF(\An))$ is an isomorphism. We will give a short proof in section~\ref{proof1.sec}.

Recall that the {\it divergence} of a vector field $\delta = \sum_{i} f_{i}\dxi$ is defined by $\Div\delta := \sum_{i} \frac{\partial f_{i}}{\partial x_{i}}$. This allows to  define the following subspaces of $\VF(\An)$:
$$
\VFzero(\An):= \{\delta\in\VF(\An) \mid \Div\delta=0\} \subset \VFc(\An):= \{\delta\in\VF(\An) \mid \Div\delta\in K\},
$$
which are Lie subalgebras, because $\Div[\delta,\eta] = \delta(\Div\eta) - \eta(\Div\delta)$. We have
$$
\VFc(\An) = \VFzero(\An) \oplus KE \text{ where } E := \sum_{i}x_{i} \dx{i} \text{ is the Euler field}.
$$
\begin{rem}
The group $\Aut(\An)$ has the structure of an {\it ind-group}, i.e. an {\it infinite dimensional algebraic group} in the sense of  \name{Shafarevich} (see \cite{Sh1966On-some-infinite-d,Sh1981On-some-infinite-d}, cf. \cite{Ku2002Kac-Moody-groups-t}). One can show that its  Lie algebra is canonically isomorphic to $\VF^{c}(\An)$. This is one of the reasons for studying this Lie algebra and its properties.
\end{rem}

 The aim of  this note is to prove the following result about the automorphism groups of these Lie algebras. 
\begin{mthm}\lab{main.thm}
There are canonical isomorphisms 
$$
\Aut(\An) \simto \AutL(\VF(\An)) \simto \AutL(\VFc(\An)) \simto  \AutL(\VFzero(\An)).
$$
\end{mthm}
\begin{rem} 
It is easy to see that the theorem holds for any field $K$ of characteristic zero. In fact, all the homomorphisms are defined over $\QQ$, and are equivariant with respect to the obvious actions of the Galois group $\Gamma=\Gal(\bar K/K)$.
\end{rem}
As a consequence, we will get the following result  which is due to \name{Kulikov}, see Corollary~\ref{cor1}.
\begin{cor*}\lab{cor1}
If every injective endomorphism of the Lie algebra $\VF(\An)$ is an automorphism, then the Jacobian Conjecture holds in dimension $n$.
\end{cor*}
\begin{rem}
It was proved by \name{Belov-Kanel} and \name{Yu} that every automorphism of $\Aut(\An)$ as an ind-group is inner (see \cite{BeYu2012Lifting-of-the-aut}). Using the main results here one can give a short of this and extend it to the closed subgroup $\SAut(\An)\subset\Aut(\An)$ of automorphism with Jacobian determinant euqal to 1, see \cite{KrZa2013Locally-finite-gro}.
\end{rem}
We add here a lemma which will be used later on. The first statement is in \cite[Lemma~3]{Sh1981On-some-infinite-d}, and from that the second follows immediately.
\begin{lem}\lab{veczero.lem}
$\VFzero(\An)$ is a simple Lie algebra, and  $\VFzero(\An) = [\VFc(\An),\VFc(\An)]$.
\end{lem}

\par\medskip
\section{Group actions and vector fields}
If an algebraic group $G$ acts on an affine variety $X$ we obtain a canonical map $\Lie G \to \VF(X)$ in the usual way (cf. \cite[II.4.4]{Kr2011Algebraic-Transfor}).
For every $A \in \Lie G$ the associated vector field $\xi_{A}$ on $X$ is defined by
\beq\lab{equ1}
(\xi_{A})_{x} := d\mu_{x}(A) \text{ for } x\in X
\eeq
where $\mu_{x}\colon G \to X$, $g\mapsto gx$,  is the orbit map in $x\in X$. It is well-known that the linear map $A \mapsto \xi_{A}$ is a anti-homomorphism of Lie algebras, and that the kernel is equal to the Lie algebra of the kernel of the action $G\to \Aut(X)$. In particular, for any algebraic subgroup $G \subset \Aut(\An)$ we have an injection $\Lie G \to \VF(\An)$. We will denote the image by $L(G)$. 
Let us point out that a connected $G\subset \Aut(\An)$ is determined by $L(G)$, i.e. if $L(G)=L(H)$ for algebraic subgroups $G,H \subset \Aut(\An)$, then $G^{0}=H^{0}$.

Recall that the vector field $\delta\in\VF(\An)$ is called {\it locally nilpotent} if the action of $\delta$ on $\Kx$ is locally nilpotent, i.e., for any $f\in\Kx$ we have 
$\delta^{m}(f) = 0$ if $m$ is large enough. Every such $\delta$ defines an action of the additive group $K^{+}$ on $\An$ such that $\delta = \xi_{1}$ where  $1 \in K = \Lie K^{+}$ (see (\ref{equ1}) above).

\begin{lem}\lab{unipotent.lem}
Let $\u\subset\VF(\An)$ be a finite dimensional commutative Lie subalgebra consisting of locally nilpotent vector fields. Then there is a commutative unipotent algebraic subgroup $U \subset \Aut(\An)$ such that $L(U)=\u$. If $\cent_{\VF(\An)}(\u)=\u$, then $U$ acts transitively on $\An$.
\end{lem}
\begin{proof}
It is clear that $\u = L(U)$ for a commutative unipotent subgroup $U\subset \Aut(\An)$. In fact, choose a basis $(\delta_{1},\ldots,\delta_{m})$ if $\u$ and consider the corresponding actions $\rho_{i}\colon K^{+}\to \Aut(\An)$. Since the associated vector fields $\delta_{i}$ commute, the same holds for the actions $\rho_{i}$, so that we get an action of $(K^{+})^{m}$. It follows that the image $U\subset\Aut(\An)$ is a commutative unipotent subgroup with $L(U)=\u$.

Assume that the action of $U$ is not transitive. Then all orbits have dimension $<n$, because orbits under unipotent groups are closed. But then there is a non-constant  $U$-invariant function $f \in\Kx$. This implies that for every $\delta\in\u$ the vector field $f\delta$ commutes with $\u$ and thus belongs to $\cent_{\VF(\An)}(\u)$, contradicting the assumption.
\end{proof}

Any $\delta\in\VF(\An)$ acts on the functions $\Kx$ as a derivation, and on the Lie algebra $\VF(\An)$ by the adjoint action, $\ad(\delta) \mu:=[\delta,\mu]$. These two actions are related as shown in the following lemma whose proof is obvious.

\begin{lem}\lab{adjoint.lem}
Let $\delta,\mu \in\VF(\An)$ be two commuting vector fields. Then
$$
\ad(\delta) (f\mu) = \delta(f) \mu.
$$
In particular, if $\ad(\delta)$ is locally nilpotent on $\VF(\An)$, then $\delta$ is locally nilpotent.
\end{lem}

\par\medskip
\section{Proof of the Main Theorem, part I}\lab{proof1.sec}
We first give a proof of the following result due to \name{Bavula} \cite{Ba2013The-group-of-autom}.
\begin{thm}\lab{thm1}
The canonical map $\Ad\colon\Aut(\An) \to \AutL(\VF(\An))$ is an isomorphism.
\end{thm}
Denote by $\Aff_{n} \subset \Aut(\An)$ the closed subgroup of affine transformations and by $S =(K^{+})^{n} \subset \Aff_{n}$ the subgroup of translations. Then 
\beq
L(\Aff_{n}) = \langle x_{i}\dd{j},\dd{k} \mid 1\leq i,j,k\leq n\rangle \supset
L(S) = \langle \dd{1},\ldots,\dd{n} \rangle.
\eeq
where $\dd{j} := \dx{j}$.
Put $\aff_{n}:=\Lie\Aff_{n}$ and $\saff_{n}:= [\aff_{n},\aff_{n}]$. We have $\saff_{n}:=\Lie \SAff_{n}$ where $\SAff_{n}:=(\Aff_{n},\Aff_{n}) \subset \Aff_{n}$ is the commutator subgroup, i.e.  the closed subgroup of those affine transformations $x\mapsto gx + b$ where $g \in \SL_{n}$.  The next lemma is certainly known. For the convenience of the reader we indicated a short proof. 
\begin{lem}\lab{AutAff.lem}
The canonical homomorphisms 
$$
\begin{CD}
\Aff_{n} @>\Ad>\simeq> \AutL(\aff_{n}) @>\res>\simeq> \AutL(\saff_{n})
\end{CD}
$$ 
are isomorphisms.
\end{lem}
\begin{proof}
It is clear that the two homomorphisms $\Ad\colon \Aff_{n} \to \AutL(\aff_{n})$ and  $\res\colon\AutL(\aff_{n}) \to \AutL(\saff_{n})$ are both injective. Thus it suffices to show that the composition $\res\circ\Ad$ is surjective.

We write the elements of $\Aff_{n}$ in the form $(v,g)$ with $v\in S = (K^{+})^{n}$, $g \in \GL_{n}$ where $(v,g)x = gx + v$ for $x\in \An$. It follows that  $(v,g)(w,h) = (v+gw, gh)$. Similarly, $(a,A)\in\aff_{n}$ means that $a\in \s=\Lie S = (K)^{n}$, $A \in \gln$, and $(a,A) x = Ax + a$. For the adjoint representation of $g\in \GL_{n}$ and of $v\in S$ on $\aff_{n}$ we find
\beq\lab{equ3}
\Ad(g)(a,A) = (ga, gAg^{-1}) \text{ \ and \ } \Ad(v)(a,A) = (a-Av,A),
\eeq
and thus, for $(b,B)\in\aff_{n}$,
\beq\lab{equ4}
\ad(B)(a,A) = (Ba, [B,A]) \text{ \ and \ } \ad(b)(a,A) = (a-Ab,A).
\eeq
Now let $\theta$ be an automorphism of the Lie algebra $\saff_{n}$. Then $\theta(\s) = \s$ since $\s$ is the solvable radical of $\saff_{n}$. Since $g := \theta|_{\s}\in\GL_{n}$, we can replace $\theta$ by $\Ad(g^{-1})\circ\theta$ and thus assume, by $(\ref{equ3})$,  that $\theta$ is the identity on $\t$. This implies that $\theta(a,A) = (a+\ell(A),\bar\theta(A))$ where $\ell\colon \sln \to \s$ is a linear map and $\bar\theta\colon\sln \simto \sln$ is a Lie algebra automomorphism.

From $(\ref{equ4})$ we get $\ad(b,B)(a,0) = \ad(B)(a,0) = (Ba,0)$ for all $a\in\s$, hence 
\begin{multline*}
(Ba,0) = \theta(Ba,0) = \theta(\ad(B)(a,0)) =\\ = \ad(\theta(B))(a,0)=\ad(\bar\theta(B))(a,0) = (\bar\theta(B)a,0).
\end{multline*}
Thus $\bar\theta(B) = B$, i.e. $\theta(a,A) = (a+\ell(A),A)$. Now an easy calculation shows that $\ell([A,B])= A\ell(B) - B\ell(A)$. This means that  $\ell$ is a cocycle of $\sln$. Since $\sln$ is semisimple, $\ell$ is a coboundary and thus $\ell(A) = Av$ for a suitable $v\in K^{n}$. In view of $(\ref{equ4})$ this implies that $\theta = \Ad(-v)$, and the claim follows.
\end{proof}

\begin{proof}[Proof of Theorem~\ref{thm1}]
It is clear that the homomorphism 
$$
\Ad\colon\Aut(\An) \to \AutL(\VF(\An))
$$ 
is injective.
So let $\theta \in \AutL(\VF(\An))$ be an arbitrary automorphism. 

We have seen above that $L(S) = \langle \dd{1},\ldots,\dd{n} \rangle \subset \VF(\An)$ where $S\subset \Aff_{n}$ is the subgroup of translations.
Since $\VF(\An) = \Kx L(S)$ we get from Lemma~\ref{adjoint.lem} that the adjoint action of any $\delta\in L(S)$ on $\VF(\An)$ is locally nilpotent, and the same holds for any element from $\u:=\theta(L(S))$. This implies, by Lemma~\ref{unipotent.lem}, that $\u = L(U)$ for a commutative unipotent subgroup $U$ of dimension $n$. Moreover, $\cent_{\VF(\An)}(L(S)) = L(S)$, hence $\cent_{\VF(\An)}(\u)=\u$ which implies, again by Lemma~\ref{unipotent.lem}, that $U$ acts transitively on $\An$. Thus every orbit map $U \to \An$ is an isomorphism. It follows that there is an automorphism $\phi\in\Aut(\An)$ such that $\phi U \phi^{-1} = S$. In fact, fix a group isomorphism $\phi\colon U \simto S$ and take the orbit maps $\mu_{S}\colon S \simto \An$ and $\mu_{U}\colon U \simto \An$ at the origin $0\in\An$. Then $\phi:= \mu_{S}\circ\phi\circ\mu_{U}^{-1}$ has the property that $\phi^{-1}u\phi = \phi(u)$ for all $u\in U$. 

It follows that the automorphism $\theta':=\Ad(\phi)\circ\theta\in\AutL(\VF(\An))$ sends $L(S)$ isomorphically onto itself. Now the relations $[\dx{i},x_{j}\dx{k}] = \delta_{ij} \dx{k}$ imply that $\theta'(L(\Aff_{n})) = L(\Aff_{n})$, and from Lemma~\ref{AutAff.lem} we obtain an affine automorphism $\tau\in\Aff_{n}$ such that $\Ad(\tau)\circ\theta'$ is the identity on $L(\Aff_{n})$. Hence, by the following lemma, $\Ad(\tau)\circ\theta'=\id$, and the claim follows.
\end{proof}

\begin{lem}\lab{endoVec.lem}
Let $\theta$ be an injective endomorphism of one of the Lie algebras $\VF(\An)$, $\VFc(\An)$ or $\VFzero(\An)$. If $\theta$ is the identity on $L(\SL_{n})$, then $\theta = \Ad(\lambda E)$ for some $\lambda \in K^{*}$.
\end{lem}

\begin{proof}
We consider the action of $\GL_{n}$ on $\VF(\An)$. Denote by $\VF(\An)_{d}$ the homogeneous vector fields of degree $d$, i.e. 
$$
\VF(\An)_{d} := \bigoplus_{i} \Kx_{d+1}\,\dd{i} \simeq \Kx_{d+1}\otimes K^{n}.
$$
Note that $\lambda E\in \GL_{n}$ acts by scalar multiplication with $\lambda^{-d}$ on $\VF(\An)_{d}$.
We have split exact sequences of $\GL_{n}$-modules
$$
\begin{CD}
0 @>>> \VFzero(\An)_{d} @>>> \VF(\An)_{d} @>\Div>> \Kx_{d+1} @>>> 0
\end{CD}
$$
where all $\SL_{n}$-modules $\VFzero(\An)_{d}$ and $\Kx_{d+1}$ are simple and pairwise non-isomorphic (see \name{Pieri}'s formula \cite[Chap. 9, section 10.2]{Pr2007Lie-groups}). 

Now let $\theta$ be an automorphism of $\VF(\An)$. If $\theta$ is the identity on $L(\SL_{n})$, then $\theta$ is $\SL_{n}$-equivariant and thus acts with a scalar $\lambda_{d}$ on $\VFzero(\An)_{d}$ and with a scalar $\mu_{d}$ on $\Kx_{d+1}$. The relation 
$$
[x_{j}^{e+1}\dd{i},x_{i}^{d+1}\dd{j}] = (d+1)x_{i}^{d}x_{j}^{e+1}\dd{j}-(e+1)x_{i}^{d+1}x_{j}^{e}\dd{j}
$$ 
shows that $\lambda_{e}\lambda_{d}=\lambda_{e+d}$, hence $\lambda_{d}=\lambda^{d}$ for a suitable $\lambda\in K^{*}$. It follows that the composition  
$\theta':=\Ad(\lambda E) \circ \theta$ is the identity on $\VFzero(\An)$. Now we use the Euler field $\ef$ and the relation $[\ef,\delta] = d \cdot\delta$ for $\delta \in \VF(\An)_{d}$ to see that $\theta'$ is the identity everywhere. This proves the claim for $\VF(\An)$. The two other cases are similar. 
\end{proof}

\par\medskip
\section{\'Etale Morphisms and Vector Fields}
In the first section we defined the action of $\Aut(\An)$ on the vector fields $\VF(\An)$ by the usual formula $\Ad(\phi)\delta:= {\phi^{*}}^{-1} \circ \delta \circ \phi^{*}$. 
In more geometric terms, considering $\delta$ as a section of the tangent bundle $T\An=\An\times\CC^{n} \to \An$, one defines the pull-back of $\delta$ by
\beq\lab{equ5}
\phi^{*}(\delta) := (d\phi)^{-1} \circ \delta \circ \phi, \text{ i.e., } \phi^{*}(\delta)_{a}= (d\phi_{a})^{-1}(\delta_{\phi(a)}) \text{ for } a \in \An.
\eeq
Clearly,  $\phi^{*}(\delta) = \Ad(\phi^{-1})\delta$. However, 
the second formula above shows that the pull-back $\phi^{*}(\delta)$ of a vector field is also defined for \'etale morphisms $\phi\colon\An \to \An$.

\begin{prop}
Let $\phi\colon \An \to \An$ be an \'etale morphism. For any vector field $\delta \in \VF(\An)$ there is a uniquely defined vector field $\phi^{*}(\delta)$ such that
\beq\lab{pullback}
d\phi \circ \phi^{*}(\delta) = \delta \circ \phi.
\eeq
The map $\phi^{*}\colon \VF(\An) \to \VF(\An)$ is an injective homomorphism of Lie algebras. Moreover, $(\eta\circ\phi)^{*}= \phi^{*}\circ\eta^{*}$.
\end{prop}
\begin{proof}
For a vector field $\xi\colon \An \to T\An$ and $a\in\An$  we have 
$(d\phi \circ \xi)_{a}=\Jac(\phi)_{a} \cdot \xi_{a}$. Thus, for a given $\delta\in\VF(\An)$, the equation $(d\phi \circ \tilde\delta)_{a} = (\delta \circ \phi)_{a}= \delta_{\phi(a)}$ has the unique solution
\[
\tilde\delta_{a}:= (\Jac(\phi)_{a})^{-1}\cdot \delta_{\phi(a)}.
\]
Since the Jacobian determinant $\det(\Jac(\phi))$ is a non-zero constant, the inverse matrix $\Jac(\phi)^{-1}$ has entries in $\Kx$. Therefore, the vector field $\tilde\delta$ is polynomial. This proves the first claim of the proposition, and the others follow immediately.
\end{proof}
\begin{rem}\lab{etale.rem}
In the notation of the proposition above, let $\phi=(f_{1},\ldots,f_{n})$ and  put $\delta = \di$ in equation~(\ref{pullback}). Then
$$
\di =  \sum_{j}\frac{\partial f_{k}}{\partial x_{j}}\,\phi^{*}(\dd{j}),
$$
Applying both sides to $f_{k}$, we get $\phi^{*}(\dd{j})(f_{k})= \delta_{jk}$.
\end{rem}
\begin{prop}
Let $\eta\colon \An \to \An$ be an \'etale morphism. If
$\eta^{*}\colon \VF(\An) \to \VF(\An)$ is an isomorphism, then $\eta$ is an automorphism of $\An$.
\end{prop}
\begin{proof} Theorem~\ref{thm1} implies that $\eta^{*}= \Ad(\phi)$ for some automorphism $\phi$ of $\An$. Since $\phi^{*}= \Ad(\phi^{-1})$, it follows that $\psi:=\phi\circ\eta$ is \'etale and that $\psi^{*}$ is the identity on $\VF(\An)$. We claim that this implies that $\psi = \id$ which will prove the proposition. 

By definition, we have $\delta_{a}= (\Jac(\psi)_{a})^{-1}\cdot \delta_{\psi(a)}$ for every vector field $\delta$. 
Putting $\delta = \dd{i}$ for $i=1,\ldots,n$, we get $\Jac(\psi)_{a} = E$ for all $a\in\An$ which implies that $\psi$ is a translation.
\end{proof}

As a corollary of the two propositions above, we get the following result which is due to \name{Kulikov} \cite[Theorem~4]{Ku1992Generalized-and-lo}.

\begin{cor}\lab{cor1}
If every injective endomorphism of the Lie algebra $\VF(\An)$ is an automorphism, then the Jacobian Conjecture holds in dimension $n$.
\end{cor}

We finish this section by showing that if the divergence of a vector field is a constant, then it does not change under an \'etale morphism. More general, we have the following result.
\begin{prop}\lab{etaleVFzero.prop}
Let $\eta\colon\An \to \An$ be an \'etale morphism, and $\delta$ a vector field. Then $\Div\eta^{*}(\delta) = \eta^{*}(\Div\delta)$. In particular, 
$\delta \in \VFc(\An)$ if and only if $\eta^{*}(\delta)\in\VFc(\An)$, and in this case we have  $\Div \eta^{*}(\delta) = \Div \delta$.
\end{prop}
\begin{proof}
Set
$\eta=(f_{1},\ldots,f_{n})$, $\delta = \sum_{j}h_{j}\dx{j}$ and $\eta^{*}(\delta) = \sum_{j}\tilde h_{j}\dx{j}$. Then
$$
h_{k}(f_{1},\ldots,f_{n}) = \sum_{i}\tilde h_{i}\frac{\partial f_{k}}{\partial x_{i}} \text{ for } k = 1,\ldots,n.
$$
Applying $\dx{j}$ to the left hand side we get the matrix
$$
\left(\sum_{i}\frac{\partial h_{k}}{\partial x_{i}}(f_{1},\ldots,f_{n}) \frac{\partial f_{i}}{\partial x_{j}}\right)_{(k,j)} = H(f_{1},\ldots,f_{n})\cdot\Jac(\eta)
$$
where $H: = \Jac(h_{1},\ldots,h_{n})$. On the right hand side, we obain similarly
$$
\left(\sum_{i}\frac{\partial \tilde h_{i}}{\partial x_{j}} \frac{\partial f_{k}}{\partial x_{i}} + 
\sum_{i}\tilde h_{i} \frac{\partial^{2}f_{k}}{\partial x_{i}\partial x_{j}}\right)_{(k,j)}=
\tilde H \cdot \Jac(\eta) + \sum_{i} \tilde h_{i} \dx{i} \Jac(\eta) 
$$
Multiplying this matrix equation from the right with $\Jac(\eta)^{-1}$ we finally get
$$
H(f_{1},\ldots,f_{n}) = \tilde H + \sum_{i} \tilde h_{i} \dx{i} \Jac(\eta)\cdot\Jac(\eta)^{-1}
$$
Now we take on both sides the traces. Using Lemma~\ref{strange.lem} below and the obvious equalities
$\Div \delta = \tr H$ and $\Div \tilde \delta = \tr \tilde H$, we finally get
$$
 \Div \tilde\delta = (\Div \delta)(f_{1},\ldots,f_{n}) = \eta^{*}(\Div\delta).
$$
The claim follows.
\end{proof}
\begin{lem}\lab{strange.lem}
Let $A$ be an $n\times n$ matrix whose entries $a_{ij}(t)$ are differentiable function in $t$. Then
$$
\tr \left( \frac{d}{dt} A \cdot \Adj(A)\right) = \frac{d}{dt} \det A.
$$
\end{lem}
\noindent
The proof is a nice exercise in linear algebra which we leave to the reader!

\par\medskip
\section{Proof of the Main Theorem, part II}
We have seen that the canonical map $\Ad\colon\Aut(\An) \to \AutL(\VF(\An))$ is an isomorphism (Theorem~\ref{thm1}). It follows from Proposition~\ref{etaleVFzero.prop} that every automorphism of $\VF(\An)$ induces an automorphism of $\VFc(\An)$. Moreover, since $\VFzero(\An) = [\VFc(\An),\VFc(\An)]$ by Lemma~\ref{veczero.lem}, we get a canonical map $\AutL(\VFc(\An)) \to \AutL(\VFzero(\An))$ which is easily seen to be injective. Thus the main theorem from section~1 follows from the next result.
\begin{thm}\lab{thm2}
The canonical map $\Ad\colon\Aut(\An) \to \AutL(\VFzero(\An))$ is an isomorphism.
\end{thm}
The proof needs some preparation. The next proposition is a reformulation of some results from \cite{No1986Commutative-bases-} and \cite{LiDu2012Pairwise-commuting}. For the convenience of the reader we will give a short proof.
\begin{prop}\lab{CommutingBasis.prop}
Let $\xi_{1},\ldots,\xi_{n} \in \VF(\An)$ be pairwise commuting and $K$-linearly independent vector fields. Then the following are equivalent.
\be
\item[(i)] There is an \'etale morphism $\eta\colon \An \to \An$ such that $\eta^{*}(\di) = \xi_{i}$ for all $i$;
\item[(ii)] $\VF(\An) = \bigoplus_{i} \Kx \xi_{i}$;
\item[(iii)] There exist $f_{1},\ldots,f_{n}\in \Kx$ such that $\xi_{i} (f_{j})=\delta_{ij}$;
\item[(iv)] $\xi_{1},\ldots,\xi_{n}$ do not have a common \name{Darboux} polynomial.
\ee
\end{prop}
\noindent
Recall that a common \name{Darboux} polynomial of the $\xi_{i}$ is a non-constant $f \in \Kx$ such that for all $i$ we have $\xi_{i}(f) = h_{i}f$ for some $h_{i}\in \Kx$.
\begin{proof} 
(a) It follows from Remark~\ref{etale.rem} that (i) implies (ii) and (iii). Clearly, (ii) implies (iv) since a common \name{Darboux} polynomial for the $\xi_{i}$ is also a common \name{Darboux} polynomial for the $\dd{i}$ which does not exist.
\par\smallskip
(b) We now show that (ii) implies (i), hence (iii), using the following well-known fact. If $h_{1},\ldots,h_{n}\in \Kx$ satisfy the conditions $\frac{\partial h_{i}}{\partial x_{j}} = \frac{\partial h_{j}}{\partial x_{i}}$ for all $i,j$, then there is an $f \in \Kx$ such that $h_{i}=\frac{\partial f}{\partial x_{i}}$ for all $i$.

By (ii) we have $\dd{i} = \sum_{k}h_{ik}\xi_{k}$ for $i=1,\ldots,n$. The relations $[\dd{i},\dd{j}]=0$ imply that $\fp{h_{ik}}{j}=\fp{h_{jk}}{i}$ for all ${i,j,k}$, hence $h_{ik}= \fp{f_{k}}{i}$ for suitable $f_{1},\ldots,f_{n}\in\Kx$. It is clear that the matrix $(h_{ik})$ is invertible, hence the morphism $\phi:=(f_{1},\ldots,f_{n})\colon \An \to \An$ is \'etale, and the claim follows from Remark~\ref{etale.rem}.
\par\smallskip
(c) Assume that (iii) holds. Setting $\xi_{i} = \sum_{k} h_{ik}\dd{k}$ and applying both sides to $f_{j}$, we see that the matrix $(h_{ik})$ is invertible, hence (ii). Thus the first three statements of the  proposition are equivalent.
\par\smallskip
(d) Finally, assume that (iv) holds. Put $\xi_{i} = \sum_{k} h_{ik}\dd{k}$. Since $[\xi_{i},\xi_{j}]=0$ we get $\xi_{i}(h_{jk}) = \xi_{j}(h_{ik})$ for all $i,j,k$. Now an easy calculation shows that $\xi_{k}(\det(h_{ij})) = \Div(\xi_{k}) \det(h_{ij})$, and so $\det(h_{ij})\in K$. If $\det(h_{ij})\neq 0$, then (ii) follows. 

If $\det(h_{ij})= 0$, then $\rk (\sum_{i=1}^{n}\Kx\xi_{i}) = r<n$, and we can assume that $\rk (\sum_{i=1}^{r}\Kx\xi_{i}) = r$. Choose a non-trivial relation $\sum_{i=1}^{r+1}f_{i}\xi_{i}=0$ where $\gcd(f_{1},\ldots,f_{r+1}) = 1$. Since $0=\xi_{j}(\sum_{i=1}^{r+1}f_{i}\xi_{i})=\sum_{i=1}^{r+1}\xi_{j}(f_{i})\xi_{i}$ we see that 
$\xi_{j}(f_{i})\in\Kx f_{i}$, and since the $\xi_{j}$ are $K$-linearly independent at least one of the $f_{i}$ is not a constant, contradicting (iv).
\end{proof}
The second main ingredient for the proof is the following result.
\begin{lem}\lab{noDarboux.lem}
Let $\xi_{1},\xi_{2}\in\VFzero(\An)$ be commuting vector fields. Assume that
\be
\item $\xi_{1}$ and $\xi_{2}$ have a common \name{Darboux} polynomial $f$;
\item Each $\xi_{i}$ acts locally nilpotently on $\VFzero(\An)$.
\ee
Then $\Kx \xi_{1} + \Kx \xi_{2} \subset \VF(\An)$ is a submodule of rank $\leq 1$.
\end{lem}
\begin{proof}
We will show that there are non-zero polynomials $p_{1}, p_{2}$ such that $p_{1}\xi_{1} = p_{2}\xi_{2}$.
Set $\xi_{i}(f) = h_{i}f$. Using the formula $\Div(g\mu) = \mu(g) + g\Div(\mu)$, we see that 
$\xi : = h_{1}\xi_{2} - h_{2}\xi_{1} \in \VFzero(\An)$. Moreover, $\xi (f) = 0$, and so $f\xi \in \VFzero(\An)$. Since
$$
[\xi_{1},\xi] = [\xi_{i},  h_{1}\xi_{2}] - [\xi_{1}, h_{2}\xi_{1}] = \xi_{1}(h_{1}) \xi_{2} - \xi_{1}(h_{2}) \xi_{1},
$$
we get $(\ad\xi_{1})^{k} \xi = \xi_{1}^{k}(h_{1})\xi_{2} - \xi_{1}^{k}(h_{2})\xi_{1}$ and
$(\ad\xi_{1})^{k} (f\xi) = \xi_{1}^{k}(fh_{1})\xi_{2} - \xi_{1}^{k}(fh_{2})\xi_{1}$. Now, by assumption (b), there is a $k>0$
such that $(\ad\xi_{1})^{k} \xi = (\ad\xi_{1})^{k} (f\xi) = 0$, hence
$$
\xi_{1}^{k}(h_{1})\xi_{2} = \xi_{1}^{k}(h_{2})\xi_{1} \text{ and }\xi_{1}^{k}(fh_{1})\xi_{2} = \xi_{1}^{k}(fh_{2})\xi_{1}.
$$
Thus the claim follows except if $\xi_{1}^{k}(h_{1}) = \xi_{1}^{k}(h_{2}) = \xi_{1}^{k}(fh_{1}) = \xi_{1}^{k}(fh_{2}) = 0$. We will show that this leads to a contradiction. Since $\xi_{1}(f)=h_{1}f$, we get
$\xi_{1}^{k+1}(f)=0$. Now choose $r,s\geq0$ minimal with $\xi_{1}^{r+1}(h_{1})=0$ and $\xi_{1}^{s+1}(f)=0$. Then $\xi_{1}^{r+s}(h_{1} f) = \xi_{1}^{r} (h_{1}) \cdot\xi_{1}^{s}(f)\neq 0$. On the other hand we have  $\xi_{1}^{s}(h_{1}f) = \xi_{1}^{s+1}(f) = 0$, a contradiction. 
\end{proof}
Now we can prove the main result of this section.
\begin{proof}[Proof of Theorem~\ref{thm2}]
Let $\theta$ be an automorphism of $\VFzero(\An)$ as a Lie algebra, and put $\xi_{i}:=\theta(\di)$. Then the vector fields $\xi_{1},\ldots,\xi_{n}$ are commuting and $K$-linearly independent. Since every $\di$ acts locally nilpotently on $\VFzero(\An)$ the same holds for each $\xi_{i}$. 

We claim that the $\xi_{i}$ do not have a common \name{Darboux} polynomial. Otherwise, Lemma~\ref{noDarboux.lem} implies that $\sum_{i}\Kx \, \xi_{i}\subset \VF(\An)$ has rank 1, i.e., there exist  $\xi\in\VF(\An)$, $p_{i}\in\Kx\setminus K$ such that $\xi_{i} = p_{i}\xi$. We can assume that $\xi$ is minimal, i.e., that $\xi$ is not of the form $p\xi'$ with a non-constant polynomial $p$. 

For every $\mu$ commuting with one of the $\xi_{i}$, we get $0=[\mu,\xi_{i}] = [\mu,p_{i}\xi] = \mu(p_{i}) \xi + p_{i}[\mu,\xi]$, hence $[\mu,\xi]\in \Kx\xi$, because $\xi$ is minimal. Since $\VFzero(\An)$ is generated, as a Lie algebra, by elements commuting with one of the $\di$, this implies that $[\VFzero(\An), \xi]=[\theta(\VFzero(\An)), \xi] \subset \Kx \xi$. But $[\di,\xi]\in\Kx\xi$ implies that $[\di,\xi]=0$, hence $\xi=0$, a contradiction.

Now we use the implication (vi) $\Rightarrow$ (iii) of Proposition~\ref{CommutingBasis.prop} to see that there is an \'etale morphism $\eta\colon \An \to \An$ such that $\xi_{i}=\eta^{*}(\di)$ for all $i$. Similarly, there is an \'etale morphism $\eta'\colon \An \to \An$ such that ${\eta'}^{*}(\di) = \theta^{-1}(\di)$. It follows that the composition $\phi:=\eta'\circ\eta$ has the property that $\phi^{*}(\di) = \di$ for all $i$. Thus, by the following lemma, $\phi$ is a translation, hence $\eta$ is an isomorphism. It follows that $\Ad(\eta)\circ \theta$ is the identity on $\langle \dd{1},\ldots,\dd{n}\rangle$, and so, again by the following lemma, we finally see that $\theta = \Ad(\psi)$ for some $\psi\in\Aut(\An)$.
\end{proof}

\begin{lem} 
Let $\theta$ be an injective endomorphism of $\VFzero(\An)$ such that  $\theta(\di) = \di$ for all $i$. Then $\theta = \Ad(\tau)$ where $\tau\colon\An \simto \An$ is a translation. In particular, $\theta$ is an automorphism.
\end{lem}
\begin{proof}
We know that $\sum_{i} K \di = L(S)$ where $S \subset \Aff_{n}$ are the translations. Moreover, $L(\Aff_{n})$ is the normalizer of $L(S)$ in the Lie algebra $\VF(\An)$. Since $\theta(L(S)) = L(S)$ we get $\theta(L(\SAff_{n})) = L(\SAff_{n})$, and so, by Lemma~\ref{AutAff.lem}, there is an affine transformation $g$ such that $\Ad(g)|_{L(\SAff_{n})} = \theta|_{L(\SAff_{n})}$. Since $\Ad(g)$ is the identity on $L(S)$ we see that $g$ is a translation. It follows that $\Ad(g^{-1})\circ\theta$ is the identity on $L(\SL_{n})$, hence, by Lemma~\ref{endoVec.lem}, $\Ad(g^{-1})\circ\theta = \Ad(\lambda E)$ for some $\lambda \in K^{*}$. But $\lambda = 1$, because $\theta$ is the identity on $L(S)$. 
\end{proof}

\par\bigskip

\end{document}